\documentclass[11pt, reqno, dvipsnames]{amsart}
\usepackage[utf8]{inputenc}
\usepackage[english]{babel}
\usepackage{color}
\usepackage{amsmath,amsthm,amssymb,amsfonts}
\usepackage{mathrsfs}
\usepackage{tikz-cd}
\usepackage[linktocpage=true]{hyperref}
\usepackage{enumerate}
\usepackage{geometry}

\hypersetup{
	colorlinks = true,
	linkcolor = {ForestGreen},
	citecolor = {RoyalBlue},
	urlcolor = {Black}
}

\newcommand{\Z}{{\mathbb Z}}
\newcommand{\Q}{{\mathbb Q}}
\newcommand{\R}{{\mathbb R}}

\newcommand{\N}{{\mathbb N}}
\newcommand{\F}{{\mathbb F}}

\newcommand{\f}{{\mathcal F}}

\newcommand{\im}{\operatorname{Im}}

\newcommand{\Spa}{\operatorname{Spa}}

\newcommand{\rg}{\operatorname{R\Gamma}}

\newcommand{\colim}{\operatorname{colim}}

\newcommand{\cond}{\operatorname{cond}}

\newcommand{\et}{\operatorname{\acute{e}t}}
\newcommand{\ket}{\operatorname{k\acute{e}t}}
\newcommand{\proet}{\operatorname{pro\acute{e}t}}
\newcommand{\proket}{\operatorname{prok\acute{e}t}}

\newcommand{\red}{\textcolor{red}}

\newcommand{\cts}{\operatorname{cts}}

\numberwithin{equation}{section}
\newtheorem{theorem}{Theorem}
\numberwithin{theorem}{section}

\newtheorem{thm}[theorem]{Theorem}
\newtheorem{lem}[theorem]{Lemma}
\newtheorem{cor}[theorem]{Corollary}
\newtheorem{prop}[theorem]{Proposition}

\theoremstyle{definition}
\newtheorem{defn}[theorem]{Definition}
\newtheorem{rem}[theorem]{Remark}

\newif\ifshowhidden
\showhiddenfalse

\AtBeginDocument{%
	\def\MR#1{}
}

\numberwithin{equation}{section}
\begin{document}
	
	\title{Cohomology of integral structure sheaves for open varieties}

	\author{Xinyu Shao}

	\address{Institut für Algebra, Zahlentheorie und Diskrete Mathematik, Leibniz Universität Hannover, Welfengarten 1, 30167 Hannover, Germany}
	
	\email{xinyu.shao@math.uni-hannover.de}

	\begin{abstract}
		We study the cohomology of integral structure sheaves for two types of open rigid analytic varieties: complements of divisors in quasi-compact varieties, and partially proper spaces. As applications, we prove a strengthened form of the Hebbarkeitssatz (Riemann's extension theorem), involving the almost vanishing of $R^1\lim$, and study pro-\'etale $G$-torsors on open varieties.
	\end{abstract}
	
	\maketitle
	\tableofcontents
	
	\section{Introduction}
	
	Let $K$ be a perfectoid field of mixed characteristic  containing all roots of unity. Let $\mathcal{O}_K$ be its ring of integers, with residue field $\kappa$ of characteristic $p$, and let $\mathfrak m$ be the maximal ideal of $\mathcal{O}_K$. We write $K^+\subseteq\mathcal O_K$ for an open integrally closed subring.

	The main goal of this article is to study the cohomology of integral structure sheaves on non-quasi-compact rigid analytic varieties and its applications (Theorems \ref{application1} and \ref{application2}). Our main technical result is the following.
	
	\begin{thm} \label{main}
		Let $X$ be a quasi-compact adic space locally of finite type over $\Spa(K,K^+)$, and let $D \subset X$ be a Zariski-closed subset of positive codimension everywhere such that $U:=X-D$ is smooth. Then for any increasing sequence $\{U_n\}_{n\in\N}$ of quasi-compact open adic subspaces over $K$ such that $U=\bigcup_n U_n$, and for any $i \geq 0$ and $m\geq 1$, the pro-system $\{H^i_{\et}(U_n,\mathcal O^+/p^m)\}_{n\in \N}$ is almost isomorphic to the constant pro-system with value $H^i_{\et}(U,\mathcal{O}^+/p^m).$
	\end{thm}

	We briefly explain the context of this result. The theory of perfectoid spaces developed by Scholze \cite{scholze2012perfectoid} provides a powerful tool for studying the cohomology of rigid analytic varieties. Together with the pro-\'etale topology, it plays a central role in $p$-adic Hodge theory \cite{scholze2013p}. An important ingredient is the cohomology of the integral structure sheaf: the almost vanishing of higher cohomology on affinoid perfectoid spaces allows one to study $\mathcal O^+/p^m$ by descent from perfectoid covers.

	For open varieties, logarithmic geometry provides a natural way to keep track of the boundary. Diao, Lan, Liu and Zhu developed the theory of log adic spaces and their Kummer \'etale and pro-Kummer \'etale topologies in \cite{dllz2023logadic}. These foundations allow one to study an open variety by equipping a suitable compactification with the log structure associated with its boundary. They also play an important role in the logarithmic Riemann--Hilbert correspondences of \cite{dllz2023logrh}. Using this framework, Lan, Liu and Zhu established de Rham comparison theorems for cohomology with compact support, as well as Poincar\'e duality for rational \'etale cohomology of smooth rigid varieties that are Zariski open in proper rigid varieties \cite{llz2023derham}.

	In the setting of Theorem \ref{main}, the open subspace $U$ is generally not quasi-compact. Computing its cohomology from an increasing sequence of quasi-compact open subspaces therefore involves derived inverse limits. The theorem shows that the resulting pro-systems of cohomology groups with coefficients in $\mathcal O^+/p^m$ are almost constant, giving control of the $R^1\lim$ terms that occur when passing to $U$. In the logarithmic setting, we can further identify their constant values with the Kummer \'etale cohomology of the ambient log adic space, yielding a comparison across the boundary.
	
	More precisely, we have the following logarithmic version of Theorem \ref{main}.
	
	\begin{thm} \label{mainlog}
		Let $X$ be a quasi-compact fs log adic space that is log smooth over $\Spa(K,K^+)$. Denote by $U:=X_{\operatorname{tr}}$ the largest open subspace of $X$ on which the induced log structure is trivial. Then for any increasing sequence $\{U_n\}_{n\in\N}$ of quasi-compact open adic subspaces over $K$ such that $U=\bigcup_n U_n$, and for any $i \geq 0$ and $m\geq 1$, the pro-system $\{H^i_{\et}(U_n,\mathcal O^+/p^m)\}_{n\in \N}$ is almost isomorphic to the constant pro-system with value $H^i_{\ket}(X,\mathcal{O}^+/p^m).$
	\end{thm}
	
	\begin{rem} \label{remequiv}
		Theorem \ref{mainlog} implies Theorem \ref{main}. In fact, by resolution of singularities, we may assume that $X$ is smooth and $D$ is a strict normal crossing divisor. Then the log structure on $X$ defined by $D$ makes $X$ log smooth. Assuming Theorem \ref{mainlog}, for any $i \geq 0$ and $m\geq 1$ we have $$\lim_n H^i_{\et}(U_n,\mathcal O^+/p^m)\simeq^a H^i_{\ket}(X,\mathcal O^+/p^m), \quad R^1\lim_n H^i_{\et}(U_n,\mathcal O^+/p^m)\simeq^a0.$$ The identity
			$$\rg_{\et}(U,\mathcal O^+/p^m) \simeq R\lim_n \rg_{\et}(U_n,\mathcal O^+/p^m)$$ gives the almost short exact sequence $$0 \to R^1\lim_n H^{i-1}_{\et}(U_n,\mathcal O^+/p^m) \to H^i_{\et}(U,\mathcal{O}^+/p^m) \to \lim_n H^i_{\et}(U_n,\mathcal O^+/p^m) \to 0.$$ It follows that for any $i \geq 0$ and $m\geq 1$ we have $$H^i_{\ket}(X,\mathcal{O}^+/p^m) \xrightarrow{\simeq^a} H^i_{\et}(U,\mathcal{O}^+/p^m),$$ which proves the claim.
	\end{rem}
	
	\begin{rem}
		Conversely, one can also deduce Theorem \ref{mainlog} from Theorem \ref{main}, but this implication is less direct. Since the proofs of the two theorems are similar, we will use the language of logarithmic geometry and prove Theorem \ref{mainlog} directly.
	\end{rem}

	\begin{rem}
		Theorem \ref{main} does not hold for perfectoid spaces: we refer the reader to Proposition \ref{failperfectoid} for a counterexample.
	\end{rem}
	
	A direct consequence of our results is a comparison between the pro-Kummer \'etale cohomology of certain period sheaves over $X$ and their pro-\'etale cohomology over $U$.
	
	\begin{cor} \label{period}
		We use the notation of Theorem \ref{mainlog}. Let $\f$ be one of the sheaves $\mathcal O^+/p^m$ ($m\geq 1$), $\widehat{\mathcal O}^+$, $\mathcal O^{\flat+}$, or $\mathbb{A}_{\inf}$ on $X_{\proket}.$ Then for $i \geq 0$, the natural map $$H^i_{\proket}(X,\f) \to H^i_{\proet}(U,\f)$$ is an almost isomorphism.
	\end{cor}

	\begin{rem}
		When $X$ is smooth and $X-D$ is a strictly normal crossing divisor, the above corollary should also be deduced from Mann's six functor formalism \cite{mann2022padic6functorformalismrigidanalytic}: the key input is that when $D$ is smooth and $j:U \hookrightarrow X$, we have $$Rj_*(\mathcal O^+_U/p) \simeq^a Rj_*(\F_p)\otimes_{\F_p}\mathcal O^+_X/p$$ by purity, deduced from the Poincar\'e duality \cite[Theorem 3.10.20]{mann2022padic6functorformalismrigidanalytic}, together with localization and the mod-$p$ Riemann–Hilbert functor \cite[Theorem 3.9.23]{mann2022padic6functorformalismrigidanalytic}. We will not discuss this approach in this article.
	\end{rem}
	
	\begin{rem}
		In general, $U$ is not quasi-compact, and the corollary does not hold for $\f[1/p]$ when $\f=\widehat{\mathcal O}^+$ or $\mathbb{A}_{\inf}$. In particular, the corollary does not recover the result of \cite[Corollary 1.3]{xsperiod}, where the author shows that the natural map $$H^i_{\proket}(X,\mathbb B_{\operatorname{dR}}) \to H^i_{\proet}(U,\mathbb B_{\operatorname{dR}})$$ is an isomorphism for any $i \geq 0$.
	\end{rem}
	

	
	One can also use Theorem \ref{mainlog} to prove the first Riemann extension theorem of \cite[\S3]{bartenwerfer1976riemannextension}, which states that if $X$ is a normal rigid space and $D$ is a divisor in $X$, then $\mathcal{O}^+(X-D)=\mathcal{O}^+(X)$. In fact, our results for $i=0$ already imply a stronger statement\footnote{As far as we know, it is not clear how to deduce $R^1\lim_n \mathcal{O}^+(U_n) \simeq^a 0$ from the proof in \cite[\S3]{bartenwerfer1976riemannextension}.}:
	
	\begin{thm}[Riemann I, Theorem \ref{riemann}] \label{application1}
		Let $X$ be a normal affinoid adic space locally of finite type over $\Spa(K,K^+)$, and let $D \subset X$ be a Zariski-closed subset of positive codimension everywhere such that $U:=X-D$ is smooth. Then for any increasing sequence $\{U_n\}_{n\in\N}$ of affinoid open adic subspaces over $K$ such that $U=\bigcup_n U_n$, we have $$\mathcal{O}^+(U)=\mathcal{O}^+(X), \quad R^1\lim_n \mathcal{O}^+(U_n) \simeq^a 0.$$
	\end{thm}
	
	\begin{rem}
		We would like to emphasize that the result $R^1\lim_n \mathcal{O}^+(U_n) \simeq^a 0$ seems to be highly nontrivial: while the problem is purely analytic in nature, our proof relies heavily on perfectoid geometry and $p$-adic Hodge theory. It would be interesting to find a more direct proof.
	\end{rem}

	\begin{rem}
		Another version of Riemann's extension theorem is the ``quantitative form of Scholze's Riemann extension theorem'' proved in \cite[Theorem 4.2]{bhatt2018directsumconj}. While we use perfectoid geometry to prove Theorem \ref{main}, it is not clear whether \cite[Theorem 4.2]{bhatt2018directsumconj} can be used to prove our results. We refer the reader to Subsection \ref{riemannperf} for a detailed discussion.
	\end{rem}
	
	The quasi-compactness condition on $X$ is not necessary if one only needs the almost vanishing of $R^1\lim$. Another result of this article is the almost vanishing of $R^1\lim$ for Stein (or, more generally, partially proper) spaces:

	\begin{thm}[Theorem \ref{stein}]
		Let $X$ be a smooth partially proper space over $K$, and let $\{U_n\}_{n\in\N}$ be an increasing sequence of quasi-compact open adic subspaces such that $X=\bigcup_n U_n$. Then for any $i \geq 0$ and $m\geq 1$, $$R^1\lim_n H^i_{\et}(U_n, \mathcal{O}^+/p^m)$$ is almost zero.
	\end{thm}
	
	The proof of this theorem uses the following result. It is somewhat surprising that no spherical completeness assumption on $K$ is needed in the almost setting.
	
	\begin{thm}[Theorem \ref{vanishingalmost}]
		Let $\{M_n\}_{n\in\N}$ be an inverse system of $\mathcal O_K$-modules. Assume that for each $n$, the image of $M_{n+1}$ in $M_n$ is almost finitely generated. Then $$R^1\lim\nolimits_n M_n \simeq^a 0.$$
	\end{thm}

	Another motivation comes from the $p$-adic Simpson correspondence, initiated by Faltings \cite{faltings2005simpson}. In \cite{heuer2024padicsimpsoncorrespondencesmooth}, Heuer constructed an equivalence between pro-\'etale vector bundles and Higgs bundles on a smooth proper rigid variety over a complete algebraically closed extension of $\Q_p$, after choosing a lift to $B_{\operatorname{dR}}^+/\xi^2$ and an exponential. An ingredient in the proof is the correspondence between small objects: \cite[Theorem 6.5]{heuer2026moduli} gives an equivalence between small $G$-bundles and small $G$-Higgs bundles after fixing a toric chart, for a rigid analytic group $G$.

	For open varieties, it is natural to ask how these small objects behave near the boundary, and how the pro-\'etale theory on the open variety compares with the pro-Kummer \'etale theory on a compactification. For a fixed toric chart, small $G$-bundles come from torsors under sufficiently small open subgroups of $G$. Such torsors can be studied using the cohomology of integral structure sheaves \cite{heuer2026gtorsor}, so our results apply to their extension across the boundary. The extension statement below also indicates a limitation of restricting to small objects on a non-quasi-compact variety: such torsors already extend to the logarithmic compactification, whereas the corresponding comparison for the full structure group $G$ can fail. We return to this point in Subsection \ref{torsors}.

	For the following statement, let $G$ be a rigid analytic group with good reduction, with a fixed smooth formal model. Let $G_k$ be its small open subgroups from Definition \ref{smallsubgroups}, and fix $k>\alpha$ in the value group of $K$, where $\alpha$ is the constant used in Subsection \ref{torsors}.
	
	\begin{thm}[Theorem \ref{gtorsor}] \label{application2}
		Let $X$ be an fs log adic space that is log smooth over $\Spa(K,K^+)$. Denote by $U:=X_{\operatorname{tr}}$ the largest open subspace of $X$ on which the induced log structure is trivial. Then any $G_k$-torsor on $U_{\proet}$ extends uniquely to a $G_k$-torsor on $X_{\proket}$.
		
		Moreover, if $X$ is affinoid or proper, and $\{U_n\}_{n\in\N}$ is an increasing sequence of affinoid open adic subspaces over $K$ such that $U=\bigcup_n U_n$, then restriction induces a bijection of pointed sets $$H^1_{\proet}(U,G_k) \xrightarrow{\sim} \lim_n H^1_{\proet}(U_n,G_k),$$ where $H^1$ denotes the set of isomorphism classes of torsors.
	\end{thm}

\begingroup
\makeatletter
\let\@tocwrite\@gobbletwo
\makeatother
	
	\subsection*{Acknowledgements} The author thanks Ben Heuer for many valuable discussions about integral structure sheaves and the $p$-adic Simpson correspondence, and for suggesting many interesting ideas that greatly improved this article. The author also thanks Lucas Mann for answering a question.

	The author acknowledges funding by the Deutsche Forschungsgemeinschaft (DFG, German Research Foundation) through the Collaborative Research Centre TRR 326 ‘Geometry and Arithmetic of Uniformized Structures’, project number 444845124.

	\subsection*{Declaration of AI usage} This article was proofread with the assistance of ChatGPT 6 Astra. The last step of the proof of Theorem \ref{vanishingalmost}, which passes from almost vanishing for an inverse system of finitely generated $\mathcal O_K$-modules to an inverse system of almost finitely generated $\mathcal O_K$-modules, was suggested by ChatGPT 5.5 Pro. No AI was used beyond these tasks.

	\subsection*{Notation and conventions}
	
	We use adic spaces as in \cite{huber2013etale} and log adic spaces as in \cite{dllz2023logadic}. Rigid analytic varieties over $K$ are assumed to be separated, taut, countable at infinity, and locally of finite type over $\Spa(K,\mathcal O_K)$. All exhaustions are by open subspaces.

	Almost mathematics is taken with respect to $\mathfrak m=p^{1/p^{\infty}}$, and $\simeq^a$ denotes an almost isomorphism. We normalize the valuation by $v(p)=1$ and write $p^r$ for a chosen element of valuation $r\in v(K^\times)$.

	We also freely use the language of log adic spaces. For definitions and properties of log adic spaces, we refer the reader to \cite{dllz2023logadic}. 

\par
\endgroup


	\section{The proof of Theorems \ref{main} and \ref{mainlog}} \label{proofmain}
	
    The aim of this section is to prove Theorems \ref{main} and \ref{mainlog}. First, we note that it suffices to prove them for any chosen sequence $\{U_n\}_{n\in\N}$, by the following simple lemma.
	
	\begin{lem} \label{choice}
		For any adic space $U=\cup_n U_n=\cup_n U_n'$ where $\{U_n\}_{n \in \N}$ and $\{U'_n\}_{n \in \N}$ are two increasing sequences of quasi-compact adic spaces with union $U$, the two ind-systems $\{U_n\}_{n \in \N}$ and $\{U'_n\}_{n \in \N}$ are equivalent. In particular, for any cohomology theory $\rg_*$ and any $i \geq 0$, the pro-systems $\{H^i_{*}(U_n)\}_n$ and $\{H^i_{*}(U'_n)\}_n$ are equivalent.
	\end{lem}
	
	\begin{proof}
		By quasi-compactness, for any $n \geq 0$, there is some $m \geq 0$ such that $U_n \subset U_m'$ and $U'_n \subset U_m$. Then the lemma follows.
	\end{proof}
	
	We also note that, as the proof of \cite[Theorem 5.1]{scholze2013p} shows, it suffices to prove Theorem \ref{mainlog} when $K^+=\mathcal{O}_K.$ After recalling some facts on pro-categories, we start with the case where $X$ is the closed ball, which is the key computation in the proof of Theorems \ref{main} and \ref{mainlog}.
	
	\subsection{Pro-categories}
	
	We recall some basic facts on pro-categories which will be used repeatedly. Let $\mathcal C$ be a category. A pro-object of $\mathcal C$ is a functor $I \to \mathcal C$, where $I$ is a small cofiltered category. We write it as $\{X_i\}_{i \in I}$. If $X=\{X_i\}_{i \in I}$ and $Y=\{Y_j\}_{j \in J}$ are pro-objects, then $$ \operatorname{Hom}_{\text{Pro}(\mathcal C)}(X,Y)=\lim_{j \in J}\varinjlim_{i \in I}\operatorname{Hom}_{\mathcal C}(X_i,Y_j). $$ This defines the category $\text{Pro}(\mathcal C)$. For $X \in \mathcal C$, we denote by $c(X)$ the associated constant pro-object.
	
	\begin{prop}
		Let $\mathcal C$ be an abelian category. Then $\text{Pro}(\mathcal C)$ is abelian, and the constant functor $c:\mathcal C \to \text{Pro}(\mathcal C)$ is exact and fully faithful.
	\end{prop}
	
	\begin{prop}
		Let $\mathcal C$ be an abelian category. The induced functor $D^+(\mathcal C) \to D^+(\text{Pro}(\mathcal C))$ is fully faithful. Equivalently, for $X,Y \in \mathcal C$ and $m \geq 0$, one has $$ \operatorname{Ext}^m_{\text{Pro}(\mathcal C)}(c(X),c(Y)) \simeq \operatorname{Ext}^m_{\mathcal C}(X,Y). $$
	\end{prop}
	
	These facts are standard; see for instance \cite[Theorem 15.3.1]{kashiwara2005categories}.
	
	\begin{rem} \label{proversuslimit}
		Being a pro-isomorphism is in general stronger than becoming an isomorphism after applying $R\lim$. For example, the natural morphism $c(\mathbb Z_p)\to \{\mathbb Z/p^n\}_{n \geq 1}$ induces an isomorphism after $R\lim$, since both derived inverse limits are $\mathbb Z_p$. However, it is not a pro-isomorphism: $\operatorname{Hom}_{\text{Pro}(\text{Ab})}(\{\mathbb Z/p^n\}_n,c(\mathbb Z_p))=\varinjlim_n \operatorname{Hom}(\mathbb Z/p^n,\mathbb Z_p)=0$.
	\end{rem}
	
	The following lemma, which shows the compatibility of constant pro-objects and spectral sequences, will be used repeatedly in the proof.
	
	\begin{lem} \label{prospectral}
		Let $\mathcal C$ be an abelian category, and let $\{M_n^\bullet\}_{n \in \mathbb N}$ be a pro-system of $\mathbb Z$-graded objects of $\mathcal C$. Fix an integer $a \geq 1$. Suppose that, for every $n \in \mathbb N$, there is a first-quadrant spectral sequence in $\mathcal C$, $$ E_{n,a}^{i,j} \Rightarrow M_n^{i+j}, $$ and that these spectral sequences, together with their abutment filtrations, are compatible with the transition morphisms of the pro-system. If, for every $i,j$, the pro-object $\{E_{n,a}^{i,j}\}_{n \in \mathbb N}$ is pro-isomorphic to a constant pro-object $c(E_a^{i,j})$, then for every $k$, the pro-object $\{M_n^k\}_{n \in \mathbb N}$ is pro-isomorphic to a constant pro-object $c(M^k)$. Moreover, the spectral sequence in $\mathcal C$ with $a$-page $E_a^{i,j}$ converges to $M^k$, and its image under $c$ is pro-isomorphic to the given pro-system of spectral sequences.
	\end{lem}
	
	\begin{proof}
		We regard each page $\{E_{n,r}^{i,j}\}_{n \in \mathbb N}$ as an object of $\text{Pro}(\mathcal C)$. By assumption, the $a$-page is constant: $ \{E_{n,a}^{i,j}\}_{n \in \mathbb N} \simeq c(E_a^{i,j}). $ Under these identifications, the differential on the $a$-page becomes a morphism $c(E_a^{i,j})\to c(E_a^{i+a,j-a+1})$. Since $c$ is fully faithful, it is uniquely of the form $c(d_a^{i,j})$ for a morphism $d_a^{i,j}:E_a^{i,j}\to E_a^{i+a,j-a+1}$ in $\mathcal C$; the identity $d_a^2=0$ follows from the faithfulness of $c$. Since $c$ is exact, passing to cohomology shows that the $(a+1)$-page is again constant. Inductively, for every $r \geq a$ and every $i,j$, one has $ \{E_{n,r}^{i,j}\}_{n \in \mathbb N} \simeq c(E_r^{i,j}), $ and the $E_r^{i,j}$ form a spectral sequence in $\mathcal C$.
		
		Fix a total degree $k$. Since the spectral sequences are first-quadrant, only the terms $E_{n,r}^{p,k-p}$ with $0 \leq p \leq k$ can contribute to $M_n^k$. Moreover, for fixed $(p,q)$, the incoming differential to $E_{n,r}^{p,q}$ is zero for $r>p$, and the outgoing differential is zero for $r>q+1$. Hence, for $p+q=k$, the term $E_{n,r}^{p,q}$ has stabilized by the page $r=\max\{a,k+2\}$, uniformly in $n$. Therefore $ \{E_{n,\infty}^{p,k-p}\}_{n \in \mathbb N} \simeq c(E_\infty^{p,k-p}) $ for every $0 \leq p \leq k$.
		
		By compatibility of the abutment filtrations, the pro-object $\{M_n^k\}_{n \in \mathbb N}$ has a finite filtration in $\text{Pro}(\mathcal C)$ whose graded pieces are the constant pro-objects $c(E_\infty^{p,k-p})$. By the derived full faithfulness recalled above, extensions of constant pro-objects come from extensions in $\mathcal C$; hence finite iterated extensions of constant pro-objects are constant. It follows that there exists $M^k \in \mathcal C$ such that $ \{M_n^k\}_{n \in \mathbb N} \simeq c(M^k). $ The induced filtration on $M^k$ has graded pieces $E_\infty^{p,k-p}$, so the spectral sequence in $\mathcal C$ with $a$-page $E_a^{i,j}$ converges to $M^k$. Applying $c$ recovers the original pro-system of spectral sequences up to pro-isomorphism.
	\end{proof}
	
	\subsection{The key computation} \label{disc}
	
	In this subsection, we consider the closed ball
	$$X=\mathbb B^d:=\Spa(R,R^{\circ}):=\Spa(K\langle T_1,...,T_d\rangle, \mathcal O_K\langle T_1,...,T_d\rangle),$$
	with $D$ the closed subset defined by $T_1\cdots T_d=0.$ In this case, we can describe the transition map $H^i_{\et}(U_{n+k},\mathcal O^+/p^m) \to H^i_{\et}(U_n,\mathcal O^+/p^m)$ very explicitly.
	
	We endow $X$ with the log structure given by the divisor $D$. Equivalently, $X$ is endowed with the log structure associated to the pre-log structure $\N^d \to K\langle \N^d\rangle =K\langle T_1,...,T_d\rangle$, where we choose a basis $\{e_1,...,e_d\}$ of $\N^d$ and send $e_i$ to $T_i$. The log adic space $X$ admits a pro-Kummer \'etale cover as follows. For each $m \in \Z_{>0},$ let $\dfrac{1}{m}\N^d$ be the toric monoid such that $\N^d \hookrightarrow \dfrac{1}{m}\N^d$ can be identified with the multiplication map $[m]:\N^d\to \N^d.$ Let ${\Q^d_{\geq 0}}:=\colim_m \left(\dfrac{1}{m}\N^d\right)$, so that $\left(\Q^d_{\geq 0}\right)^{\operatorname{gp}}=\Q^d.$ Set $$\mathbb B_{m}:=\Spa\left(K\left\langle\dfrac{1}{m}\N^d\right\rangle,\mathcal{O}_K\left\langle\dfrac{1}{m}\N^d\right\rangle\right),$$ equipped with the log structure induced by $\dfrac{1}{m}\N^d \to K\langle\dfrac{1}{m}\N^d\rangle.$ The morphism $\mathbb B_m \to X$ is a finite Kummer \'etale cover with Galois group $$\Gamma_m=\operatorname{Hom}(\Z^d,\boldsymbol{\mu}_m)\simeq \boldsymbol{\mu}_m^d,$$ where $\boldsymbol{\mu}_m$ is the group of $m$-th roots of unity in $K$. Denote the log affinoid perfectoid object by $$\widetilde{\mathbb B}:=\lim_m \mathbb B_{m} \in X_{\proket},$$ where for $m|m'$, the transition maps are induced by $\dfrac{1}{m}\N^d \hookrightarrow \dfrac{1}{m'}\N^d.$ Then we have the associated log affinoid perfectoid space $$\widehat{\widetilde{\mathbb B}}:=\Spa(K\langle {\Q^d_{\geq 0}}\rangle,\mathcal{O}_K\langle {\Q^d_{\geq 0}}\rangle).$$ The morphism $\widehat{\widetilde{\mathbb B}} \to X$ is a Galois pro-finite Kummer \'etale cover with Galois group $$\Delta=\operatorname{Hom}({\Q^d}/\Z^d,\boldsymbol{\mu}_{\infty})\simeq \widehat{\Z}^d,$$ where $\boldsymbol{\mu}_{\infty}:=\cup_{m\in \N} \boldsymbol{\mu}_{m}.$ The natural action of the Galois group $\Delta$ on $K\langle {\Q^d_{\geq 0}}\rangle$ is given by $$\gamma(T^a)=\gamma(a)T^a$$ for all $\gamma \in \Delta$ and $a \in {\Q^d_{\geq 0}},$ where $T^a$ is the element corresponding to $a$ in $K\langle {\Q^d_{\geq 0}}\rangle$.

	Set $$\widetilde{X}:=\Spa(R_{\infty},R_{\infty}^{\circ}):=\Spa(K\langle T_1^{1/p^{\infty}},...,T_d^{1/p^{\infty}}\rangle, \mathcal O_K\langle T_1^{1/p^{\infty}},...,T_d^{1/p^{\infty}}\rangle).$$ Then $\widetilde{X}\to X$ is a Galois pro-finite Kummer \'etale cover\footnote{This is not a log affinoid perfectoid space according to the definition in \cite[Definition 5.3.1]{dllz2023logadic}, but we will see that it suffices for computing Kummer \'etale cohomology.} with Galois group $\Gamma:=\Z_p^d.$
	
	For any $n \geq 0$, set $$U_n:=\Spa(R_n,R_n^{\circ}):=\Spa(K\langle T_1,...,T_d,p^n/(T_1\cdots T_d)\rangle, \mathcal O_K\langle T_1,...,T_d,p^n/(T_1\cdots T_d)\rangle).$$
	According to Lemma \ref{choice}, it suffices to prove Theorem \ref{mainlog} for our chosen $U_n$. Denote by $\widetilde{U}_n:=\Spa(R_{n,\infty},R_{n,\infty}^{\circ})$ the pullback of $U_n$ along $\widetilde{X} \to X$. Then for $n\geq 0$ the map $\widetilde{U}_n \to U_n$ is a Galois pro-finite \'etale cover with Galois group $\Gamma.$ Since $\widetilde{U}_n$ is an affinoid perfectoid space, the Cartan--Leray spectral sequence gives
	\begin{equation} \label{groupcoh}
	\rg_{\et}(U,\mathcal O^+/p^m) \simeq R\lim_n \rg_{\et}(U_n,\mathcal O^+/p^m) \simeq^a R\lim_n \rg_{\cts}(\Gamma,R_{n,\infty}^{\circ}/p^m).
	\end{equation}
	We prove that for any $i \geq 0,$ the pro-system $\{H^i_{\cts}(\Gamma,R_{n,\infty}^{\circ}/p^m)\}_{n\in \N}$ is almost isomorphic to the constant pro-system with value $H^i_{\cts}(\Gamma,R_{\infty}^{\circ}/p^m)$ (as an $R^{\circ}/p^m$-module).
	
	For $n \geq 1$, we can decompose $R_{n,\infty}^{\circ}/p^m$ as an $\mathcal O_K/p^m$-module as follows:
	\begin{align*}
		R_{n,\infty}^{\circ}/p^m \simeq^a& \mathcal O_K/p^m\left[T_1^{1/p^{\infty}},...,T_d^{1/p^{\infty}},\left(\dfrac{p^n}{T_1T_2\cdots T_d}\right)^{1/p^{\infty}} \right]	\\
		\simeq & R_{\infty}^{\circ}/p^m \bigoplus \left(\bigoplus_{\mathbf{w}} \mathcal O_K/p^m \cdot \dfrac{p^{n\cdot \max(\mathbf{w})}}{\underline{T}^{\mathbf{w}}}\right).
	\end{align*}
	Here $\mathbf{w}:=(w_1,...,w_d)\in \Z[1/p]^d$ ranges over the vectors such that at least one of $w_1,...,w_d$ is positive, and $\max(\mathbf{w}):=\max\{w_1,...,w_d\}$. For simplicity, we write $R_{n,\mathbf{w}}:=\mathcal O_K/p^m \cdot \dfrac{p^{n\cdot \max(\mathbf{w})}}{\underline{T}^{\mathbf{w}}}$. Then $$H^i_{\cts}(\Gamma,R_{n,\infty}^{\circ}/p^m)\simeq^a H^i_{\cts}(\Gamma,R_{\infty}^{\circ}/p^m) \bigoplus \left(\bigoplus_{\mathbf{w}} H^i_{\cts}(\Gamma,R_{n,\mathbf{w}})\right).$$
	
	Fix $m\geq 1$ and set $N:=\max\{m,2\}$. We claim that for every such $\mathbf{w}$, every $n\geq 1$ and every $i\geq 0$, the transition map $$H^i_{\cts}(\Gamma,R_{n+N,\mathbf{w}}) \longrightarrow H^i_{\cts}(\Gamma,R_{n,\mathbf{w}})$$ is zero.
	In fact, put $a:=\max(\mathbf{w})>0$. Identifying the rank-one modules $R_{n+N,\mathbf{w}}$ and $R_{n,\mathbf{w}}$ by their displayed bases gives a $\Gamma$-equivariant isomorphism, under which the transition map is multiplication by a scalar of valuation $Na$. If $a\geq 1$, then $Na\geq m$, so the map is already zero on the coefficient modules.
	If $0<a<1$, choose $k$ such that $w_k=a$ and write $a=r/p^s$, where $s\geq 1$ and $(r,p)=1$. The $k$-th basis vector $\gamma_k\in\Gamma$ acts on $R_{n,\mathbf{w}}$ by a primitive $p^s$-th root of unity $\zeta$. Since continuous group cohomology is computed by the Koszul complex, multiplication by $\zeta-1$ annihilates $H^i_{\cts}(\Gamma,R_{n,\mathbf{w}})$ for every $i\geq 0$. Moreover,
	$$v(\zeta-1)=\frac{1}{p^{s-1}(p-1)}\leq\frac{2}{p^s}\leq 2a\leq Na.$$
	Thus the scalar defining the transition map is divisible by $\zeta-1$ in $\mathcal O_K$, so the induced map on cohomology is zero. This proves the claim. Since $N$ is independent of $\mathbf{w}$, the pro-system of direct sums of these additional weight contributions is zero.
	
	This shows that for any $i \geq 0,$ the pro-system $\{H^i_{\cts}(\Gamma,R_{n,\infty}^{\circ}/p^m)\}_{n\in \N}$ is almost isomorphic to the constant pro-system with value $H^i_{\cts}(\Gamma,R_{\infty}^{\circ}/p^m)$ (as an $R^{\circ}/p^m$-module). We have $$\rg_{\ket}(X,\mathcal O^+/p^m) \simeq^a \rg_{\cts}(\Delta, \mathcal{O}_K\langle \Q^d_{\geq 0}\rangle/p^m) \simeq^a \rg_{\cts}(\Gamma, R_{\infty}^{\circ}/p^m),$$ where the first almost quasi-isomorphism follows from the Cartan--Leray spectral sequence, and the second almost quasi-isomorphism follows from \cite[Lemma 3.10]{olsson2009faltings}. This completes the proof of Theorem \ref{mainlog} in the case where $X$ is the closed ball, and $D$ is defined by $T_1\cdots T_d=0$. Then Theorem \ref{main} follows by Remark \ref{remequiv}.
	
	\subsection{The case of tubular neighborhoods}
	
	In this subsection, we prove the case where $X:=\Spa(A,A^{\circ})\times_K \mathbb B^d=\Spa(A\langle T_1,...,T_d\rangle, A^{\circ}\langle T_1,...,T_d\rangle)$ for some smooth Tate algebra $A$ over $K$, and $D$ is again the closed subset of $X$ defined by $T_1\cdots T_d=0.$ As before, we endow $X$ with the log structure given by the divisor $D$. We keep the same notation for $(R,R^{\circ}),(R_{\infty},R_{\infty}^{\circ}),(R_{n},R_{n}^{\circ})$ and $(R_{n,\infty},R_{n,\infty}^{\circ})$ as in the previous subsection. Set
	\begin{equation} \label{increasingopen}
		U_n:=\Spa(A,A^{\circ}) \times_K \Spa(R_{n},R_{n}^{\circ}).
	\end{equation}
	By Lemma \ref{choice}, it suffices to prove Theorems \ref{main} and \ref{mainlog} for $U_n$.
	
	We first consider the case where there exists an \'etale map $$\Spa(A,A^{\circ}) \to \mathbb T^{d'}:=\Spa(K\langle X_1^{\pm 1},...,X_{d'}^{\pm 1}\rangle, \mathcal O_K\langle X_1^{\pm 1},...,X_{d'}^{\pm 1}\rangle)$$ which can be written as a composition of rational subsets and finite \'etale maps. Denote by $\Spa(A_n,A^{\circ}_n)$ the pullback of $\Spa(A,A^{\circ})$ along $$\mathbb T^{d'}_n:=\Spa(K\langle X_1^{\pm 1/p^n},...,X_{d'}^{\pm 1/p^n}\rangle, \mathcal O_K\langle X_1^{\pm 1/p^n},...,X_{d'}^{\pm 1/p^n}\rangle) \to \mathbb T^{d'},$$ and by $(A_{\infty},A^{\circ}_{\infty})$ the $p$-adic completion of the direct limit of $(A_n,A^{\circ}_n).$ Then $\Spa(A_{\infty},A^{\circ}_{\infty})\to \Spa(A,A^{\circ})$ is a Galois pro-finite \'etale cover with Galois group $\Gamma':=\Z_p^{d'}.$
	
	In this case, $\widetilde{X}:=\Spa(A_{\infty}\widehat{\otimes}_{K}R_{\infty},A_{\infty}^{\circ}\widehat{\otimes}_{\mathcal{O}_K}R_{\infty}^{\circ}) \to X$ is a Galois Kummer pro-finite \'etale cover with Galois group $\Gamma \times \Gamma'.$ Set $$\widetilde{U}_n:=\Spa(A_{\infty},A^{\circ}_{\infty}) \times_K  \Spa(R_{n,\infty},R_{n,\infty}^{\circ}) = \Spa(A_{\infty}\widehat{\otimes}_{K}R_{n,\infty},A_{\infty}^{\circ}\widehat{\otimes}_{\mathcal{O}_K}R_{n,\infty}^{\circ}).$$ Then for $n\geq 0$ the map $\widetilde{U}_n \to U_n$ is a Galois pro-finite \'etale cover with Galois group $\Gamma\times \Gamma'.$ By applying the Cartan--Leray spectral sequence, we have
	$$\rg_{\et}(U,\mathcal O^+/p^m) \simeq R\lim_n \rg_{\et}(U_n,\mathcal O^+/p^m) \simeq^a R\lim_n \rg_{\cts}(\Gamma\times \Gamma',R_{n,\infty}^{\circ}/p^m \otimes_{\mathcal O_K/p^m}A^{\circ}_{\infty}/p^m).$$
	On the other hand, for $i,j\geq 0$ and $n \geq 0$, we have the spectral sequence computing continuous group cohomology:
	$$E_{n,2}^{i,j}:=H^i_{\cts}(\Gamma', H^j_{\cts}(\Gamma,R_{n,\infty}^{\circ}/p^m \otimes_{\mathcal O_K/p^m}A^{\circ}_{\infty}/p^m)) \Rightarrow H^{i+j}_{\cts}(\Gamma\times \Gamma',R_{n,\infty}^{\circ}/p^m \otimes_{\mathcal O_K/p^m}A^{\circ}_{\infty}/p^m).$$
	As $\Gamma$ acts trivially on $A^{\circ}_{\infty}/p^m$, we have $$H^i_{\cts}(\Gamma', H^j_{\cts}(\Gamma,R_{n,\infty}^{\circ}/p^m \otimes_{\mathcal O_K/p^m}A^{\circ}_{\infty}/p^m))\simeq H^i_{\cts}(\Gamma', H^j_{\cts}(\Gamma,R_{n,\infty}^{\circ}/p^m )\otimes_{\mathcal O_K/p^m}A^{\circ}_{\infty}/p^m).$$
	According to our calculation in the previous subsection, the pro-system $\{H^j_{\cts}(\Gamma,R_{n,\infty}^{\circ}/p^m )\}_{n\in \N}$ is almost isomorphic to the constant pro-system with value $H^j_{\cts}(\Gamma,R_{\infty}^{\circ}/p^m )$. Therefore, for each $i,j\geq 0$,  the pro-system $\{E_{n,2}^{i,j}\}_{n\in \N}$ is almost isomorphic to the constant pro-system with value $$E_{2}^{i,j}:=H^i_{\cts}(\Gamma', H^j_{\cts}(\Gamma,R_{\infty}^{\circ}/p^m )\otimes_{\mathcal O_K/p^m}A^{\circ}_{\infty}/p^m)\simeq H^i_{\cts}(\Gamma', H^j_{\cts}(\Gamma,R_{\infty}^{\circ}/p^m \otimes_{\mathcal O_K/p^m}A^{\circ}_{\infty}/p^m)).$$
	Since we have $$E_{2}^{i,j} \Rightarrow H^{i+j}_{\cts}(\Gamma\times \Gamma',R_{\infty}^{\circ}/p^m \otimes_{\mathcal O_K/p^m}A^{\circ}_{\infty}/p^m) \simeq^a H^{i+j}_{\ket}(X,\mathcal O^+/p^m),$$ by Lemma \ref{prospectral}, we have that for each $i\geq 0$ and $m \geq 1$, the pro-system $\{H^{i}_{\cts}(\Gamma\times \Gamma',R_{n,\infty}^{\circ}/p^m \otimes_{\mathcal O_K/p^m}A^{\circ}_{\infty}/p^m)\}_{n\in \N}$ is almost isomorphic to the constant pro-system with value $H^{i}_{\cts}(\Gamma\times \Gamma',R_{\infty}^{\circ}/p^m \otimes_{\mathcal O_K/p^m}A^{\circ}_{\infty}/p^m)$, which means that the pro-system $\{H^i_{\et}(U_n,\mathcal O^+/p^m)\}_{n\in \N}$ is almost isomorphic to the constant pro-system with value $H^{i}_{\ket}(X,\mathcal O^+/p^m).$
	
	In general, if $A$ is a smooth Tate algebra over $K$, there exists an open cover of $\Spa(A,A^{\circ})$ by $\{\Spa(A_s,A_s^{\circ})\}_{s\in S}$ for a finite index set $S$, and there exists some $d' \in \N$ such that for each $s \in S$, there exists an \'etale map $$\Spa(A_s,A_s^{\circ}) \to \mathbb T^{d'}:=\Spa(K\langle X_1^{\pm 1},...,X_{d'}^{\pm 1}\rangle, \mathcal O_K\langle X_1^{\pm 1},...,X_{d'}^{\pm 1}\rangle)$$ which can be written as a composition of rational subsets and finite \'etale maps. We set $X_s:=\Spa(A_s\langle T_1,...,T_d\rangle, A^{\circ}_s\langle T_1,...,T_d\rangle), U_{s,n}:=\Spa(A_s,A_s^{\circ}) \times_K \Spa(R_{n},R_{n}^{\circ})$ for $s\in S,n \geq 0$. Let $D_s \subset X_s$ be the closed subset defined by $T_1\cdots T_d=0$ and $U_s:=X_s-D_s$. We endow $X_s$ with the log structure given by the divisor $D_s$: this is the same as the pullback of the log structure on $X$ along $X_s \to X$.
	
	Let $X^{\bullet} \to X$ (resp. $U_n^{\bullet} \to U_n$, resp. $U^{\bullet} \to U$) be the \v Cech nerve associated to the covering $\{X_s\}_{s\in S}$ (resp. $\{U_{s,n}\}_{s\in S}$, resp. $\{U_{s}\}_{s\in S}$). Then for each $n \geq 0$ the \v Cech-to-cohomology spectral sequence gives $$E_{n,1}^{i,j}:=H^j_{\et}(U_n^i,\mathcal O^+/p^m) \Rightarrow H^{i+j}_{\et}(U_n,\mathcal O^+/p^m).$$ 
	According to our previous discussion, for each $i,j \geq 0$, the pro-system $\{H^j_{\et}(U_n^i,\mathcal O^+/p^m)\}_{n\in \N}$ is almost isomorphic to the constant pro-system with value $H^j_{\ket}(X^i,\mathcal O^+/p^m)$. Since we have (again by the \v Cech-to-cohomology spectral sequence) $$E_1^{i,j}:=H^j_{\ket}(X^i,\mathcal O^+/p^m)\Rightarrow H^{i+j}_{\ket}(X,\mathcal O^+/p^m),$$ by Lemma \ref{prospectral} we have that for each $i \geq 0,$ the pro-system $\{H^i_{\et}(U_n,\mathcal O^+/p^m)\}_{n\in \N}$ is almost isomorphic to the constant pro-system with value $H^i_{\ket}(X,\mathcal O^+/p^m)$.
	This completes the proof of Theorems \ref{main} and \ref{mainlog} in the case where $X:=\Spa(A,A^{\circ})\times_K \mathbb B^d$ for some smooth Tate algebra $A$ over $K$, and $D$ is the closed subset defined by $T_1\cdots T_d=0.$
	
	\subsection{The general case} \label{descent}
	
	Now let $X$ be an fs log adic space that is log smooth over $\Spa(K,K^+)$, and let $U:=X_{\operatorname{tr}}$ be the maximal open of $X$ such that the log structure on $U$ induced by $X$ is trivial. As in our proof in the previous subsection (using the \v Cech-to-cohomology spectral sequence), we can prove Theorem \ref{mainlog} \'etale locally (with induced log structure). By resolution of toric singularities \cite[Theorem 1.2]{xslogdescent} and log \'etale descent \cite[Theorem 1.1]{xslogdescent}, we may assume that $X$ is smooth and that $D:=X-U$ is a strict normal crossing divisor\footnote{By our discussion, one can in fact assume that there exists a strictly \'etale map $X \to \mathbb B^d$ which can be written as a composition of rational subsets and finite \'etale maps. However, it seems that this additional assumption does not make the proof any easier.}. Let $ \{U_n\}_{n\in\N}$ be an increasing sequence of quasi-compact adic spaces over $K$ such that $U=\bigcup_n U_n.$
	
	Recall the following theorem of Kiehl \cite[Theorem 1.18]{kiehl1967derham} (see also \cite[Theorem 2.11]{mitsui2009criterion}), which shows the existence of tubular neighborhoods.
	
	\begin{thm} \label{tubular}
		If $D \subset X$ is a strict normal crossing divisor in a smooth rigid space, then for any (adic) point $x$ in $X$ contained in exactly $r$ irreducible components $D_1, \dots, D_r$ of $D$ (we allow $r=0$), we can find a small open affinoid $V \subset X$ containing $x$ together with a smooth affinoid $\Spa(A,A^{\circ})$ and an isomorphism $V \simeq \Spa(A,A^{\circ}) \times_K \mathbb{B}^r =\Spa(A\langle T_1,...,T_r\rangle, A^{\circ}\langle T_1,...,T_r\rangle)$, under which the individual components $D_i \cap V$ containing $x$ identify with the zero loci of the coordinate functions $T_i \in \mathcal{O}(\mathbb{B}^r)$.
	\end{thm}
	
	Therefore, we can take a hypercover $X^{\bullet} \to X$ such that for each $i\geq 0$, $X^i=\bigsqcup_{l \in S_i}V^{i,l}$ is the disjoint union of tubular neighborhoods $V^{i,l}$ of the form described in Theorem \ref{tubular}, where $S_i$ is an index set. Following the previous subsection, for each $n \geq 0$ one can define an open subset $U_n^{i,l}\subset V^{i,l}$ in the same way as in (\ref{increasingopen}). Let $D^{i,l}$ be the union of the zero loci of the coordinate functions. Then $V^{i,l}-D^{i,l}$ is the increasing union of $\{U_n^{i,l}\}_{n\in \N}$. For $i \geq 0,l\in S_i$, $V^{i,l}$ has a natural log structure induced by $X$, which can also be described as the log structure given by the strict normal crossing divisor $D^{i,l}.$ We also endow $X^i$ with the natural log structure.
	
	The hypercover $X^{\bullet} \to X$ induces a hypercover $U_n^{\bullet} \to U_n$ for each $n \geq 0.$ For each $i$, we have $U_n^i=\bigsqcup_{l \in S_i}(V^{i,l}\cap U_n)$. The hypercover-to-cohomology spectral sequence gives 
	$$E_{n,1}^{i,j}:=H^j_{\et}(U_n^i,\mathcal O^+/p^m)=\bigoplus_{l\in S_i}H^j_{\et}(V^{i,l}\cap U_n,\mathcal O^+/p^m) \Rightarrow H^{i+j}_{\et}(U_n,\mathcal O^+/p^m).$$
	For each $j \geq 0, i\geq 0$ and $l \in S_i$, Lemma \ref{choice} shows that the pro-system $\{H^j_{\et}(V^{i,l}\cap U_n,\mathcal O^+/p^m)\}_{n\in\N}$ is isomorphic to the pro-system $\{H^j_{\et}(U_n^{i,l},\mathcal O^+/p^m)\}_{n\in\N}$, which is almost isomorphic to the constant pro-system with value $H^j_{\ket}(V^{i,l},\mathcal O^+/p^m)$ by our discussion in the previous subsection. Therefore, the pro-system $\{H^j_{\et}(U_n^{i},\mathcal O^+/p^m)\}_{n\in\N}$ is almost isomorphic to the constant pro-system with value $H^j_{\ket}(X^{i},\mathcal O^+/p^m)$. Once again, since we have $$E_1^{i,j}:=H^j_{\ket}(X^i,\mathcal O^+/p^m)\Rightarrow H^{i+j}_{\ket}(X,\mathcal O^+/p^m),$$ by Lemma \ref{prospectral} we have that for each $i \geq 0,$ the pro-system $\{H^i_{\et}(U_n,\mathcal O^+/p^m)\}_{n\in \N}$ is almost isomorphic to the constant pro-system with value $H^i_{\ket}(X,\mathcal O^+/p^m)$.
	This completes the proof of Theorems \ref{main} and \ref{mainlog}.

	\section{Cohomology of integral structure sheaves for partially proper varieties}

	In this section, we study the almost vanishing of $R^1\lim$ for smooth partially proper varieties. We first establish an almost vanishing criterion for inverse systems with almost finitely generated images, and then apply it to the cohomology of integral structure sheaves. We also discuss the role of spherical completeness in obtaining strict vanishing for open balls.
	
	\subsection{Inverse systems of almost finitely generated modules}

	We recall a few standard facts on linearly compact modules which will be used to control derived inverse limits. We then generalize them to the almost setting.
	
	\begin{defn}
		We say that $K$ is spherically complete if every decreasing family of closed balls in $K$ has non-empty intersection.
	\end{defn}
	
	The following notion is the module-theoretic analogue of compactness adapted to linear topologies.
	
	\begin{defn}
		Let $R$ be a topological ring. A topological $R$-module $M$ is said to be linearly topologized if the open submodules of $M$ form a basis of neighborhoods of $0$. A Hausdorff linearly topologized $R$-module $M$ is said to be linearly compact if every family of closed cosets $\{x_i+N_i\}_{i\in I}$, where $N_i\subset M$ are closed submodules, has non-empty intersection whenever every finite subfamily has non-empty intersection.
	\end{defn}
	
	\begin{rem}
		For a non-archimedean field $K$, spherical completeness of $K$ is equivalent to linear compactness of $\mathcal O_K$ with respect to its valuation topology. Thus, if $K$ is spherically complete, then $\mathcal O_K$ is linearly compact as an $\mathcal O_K$-module.
	\end{rem}
	
	We shall use this criterion only through the following immediate consequence.
	
	\begin{lem}
		Assume that $K$ is spherically complete. Then every finitely generated $\mathcal O_K$-module is linearly compact.
	\end{lem}
	
	\begin{proof}
		Since $\mathcal O_K$ is linearly compact, the same is true for $\mathcal O_K^r$ for every $r\geq 0$. If $M$ is a finitely generated $\mathcal O_K$-module, then $M$ is a quotient of some $\mathcal O_K^r$. Since linear compactness is preserved under quotients, $M$ is linearly compact.
	\end{proof}
	
	The main point of introducing linearly compact modules is the following vanishing property for towers. It is the analogue of the Mittag-Leffler vanishing for compact objects.
	
	\begin{thm} \cite[Theorem 7.1]{jensen1972derivedlimit}
		Let $\{M_n\}_{n\in\mathbb N}$ be an inverse system of linearly compact modules with continuous transition maps. Then
		$$R^1\lim\nolimits_n M_n=0.$$
	\end{thm}
	
	Combining the preceding statements gives the form which will be used below.
	
	\begin{cor} \label{linearvanishing}
		Let $K$ be a spherically complete field, and let $\{M_n\}_{n\in\mathbb N}$ be an inverse system of finitely generated $\mathcal O_K$-modules. Then
		$$R^1\lim\nolimits_n M_n=0.$$
	\end{cor}
	
	The assumption that $K$ is spherically complete is necessary. We refer the reader to Proposition \ref{opendiscsection} for an explicit calculation. Surprisingly, no spherical completeness assumption on $K$ is needed in the almost setting.
	
	\begin{thm} \label{vanishingalmost}
			Let $K$ be any perfectoid field. Let $\{M_n\}_{n\in\N}$ be an inverse system of $\mathcal O_K$-modules. Assume that for each $n$, the image of $M_{n+1}$ in $M_n$ is almost finitely generated. Then
		$$R^1\lim\nolimits_n M_n \simeq^a 0.$$
	\end{thm}
	
	We will prove the theorem in the following steps.
	
	
	\begin{prop} \label{vanishingalmostfg}
		Let $\{M_n\}_{n\in\N}$ be an inverse system of $\mathcal O_K/p^m$-modules. Assume that for each $n$, the image of $M_{n+1}$ in $M_n$ is finitely generated. Then $$R^1\lim\nolimits_n M_n \simeq^a 0.$$
	\end{prop}
	
	\begin{proof}
		We may assume that each $M_n$ is finitely generated. We first note that, if $M$ is a finitely generated $\mathcal O_K/p^m$-module, then any decreasing chain of submodules $P_1 \supset P_2 \supset \cdots$ of $M$ is almost stationary, that is, for any $t \in \mathfrak m,$ there exists some $N>0$ such that for any $r \geq N$, $tP_N \subset P_r$. In fact, it suffices to prove the claim for $M=\mathcal O_K/p^m$: one may then use induction to prove it for $(\mathcal O_K/p^m)^n$, and hence for any finitely generated $M$. For $M=\mathcal O_K/p^m$, the $P_i$ are ideals of $M$, and we can define $\alpha_i:=\inf\{v(x)|x\in P_i\} \in [0,m]$ (we pick an arbitrary lift of $x$ to $\mathcal O_K$). Here $v(-)$ is the $p$-adic valuation. Then $(\alpha_i)$ is an increasing sequence. Let $\alpha:=\sup_i \alpha_i.$ For any $t \in \mathfrak m$, we can choose $N$ such that $\alpha_N +v(t) >\alpha$, which proves the claim.
		
		Now choose a decreasing sequence $\epsilon_n>0$ such that $\sum \epsilon_n$ converges in the value group of $K$. We construct a new inverse system $\{M_n'\}$ inductively: for each $n$, let $M_n'$ be the image of $M_N$ in $M_n,$ where $N$ is chosen such that for any $r \geq N$, $p^{\epsilon_n}\im(M_N \to M_n) \subset \im(M_r \to M_n)$. Moreover, we choose $N$ so that the image of $M_n'$ in $M_{n-1}$ is contained in $M_{n-1}'$. This gives an exact sequence of inverse systems $$0 \to \{M_n'\} \to \{M_n\} \to \{M_n/M_n'\} \to 0.$$
		
		For any $n \geq 0$, the construction of $M_n'$ shows that there exists some $N>0$ such that the map $M_N/M_N' \to M_n/M_n'$ is (strictly) zero. Therefore $R^1\lim_n M_n/M_n'=0.$ On the other hand, since we have (by construction) $$p^{\epsilon_n} M_n' \subset \im (M_{n+1}'\to M_n') \subset M_n',$$ we know that $p^{\epsilon_n}\cdot  \operatorname{coker}(M_{n+1}'\to M_n')=0$. Then \cite[Lemma 2.4.2.(iii)]{gabber2003almost} shows that $R^1\lim_n M_n'$ is almost zero. Therefore $R^1\lim\nolimits_n M_n$ is almost zero, which completes the proof.
	\end{proof}
	
	\begin{prop} \label{vanishingalmostfg2}
		Let $\{M_n\}_{n\in\N}$ be an inverse system of $\mathcal O_K$-modules. Assume that for each $n$, the image of $M_{n+1}$ in $M_n$ is finitely generated. Then $$R^1\lim\nolimits_n M_n \simeq^a 0.$$
	\end{prop}
	
	\begin{proof}
		We may assume that each $M_n$ is finitely generated. As finitely generated modules over $\mathcal O_K$ are complete, we know that $R\lim_n M_n$ is derived complete, and therefore $R^1\lim_n M_n$ is derived complete. Proposition \ref{vanishingalmostfg} shows that $R^1\lim_n( M_n/p) \simeq (R^1\lim_n M_n)/p$ is almost zero. We then use the almost derived Nakayama lemma below to conclude.
	\end{proof}
	
	The above proof uses the following almost version of the derived Nakayama lemma.
	
	\begin{lem}
		Let $M$ be an (ordinary) derived complete $\mathcal{O}_K$-module. If $M/pM$ is almost zero, then $M$ is almost zero. 
	\end{lem}
	
	\begin{proof}
		Suppose the opposite. Then there exist elements $x\in M$ and $t_0\in\mathfrak{m}$ such that $t_0x \neq 0.$ Choose a sequence $\epsilon_n>0$ such that $\sum \epsilon_n=v(t_0)$. As $\mathfrak{m}M=pM,$ we can construct a sequence $z_n$ inductively such that $z_1=x$ and $pz_{n+1}=p^{\epsilon_n}z_n$. Let $y_n:=(\prod_{i=n}^{\infty}p^{\epsilon_i})z_n.$ Then $py_{n+1}=y_n$. Therefore, as $y_1=t_0x \neq 0$, the sequence $(y_n)$ defines a nonzero element $(y_n)\in\lim(M\xleftarrow p M\xleftarrow p\cdots),$ which contradicts the derived completeness of $M.$
	\end{proof}

	\begin{proof}[Proof of Theorem \ref{vanishingalmost}]
		We may assume that each $M_n$ is almost finitely generated. We first record a simple consequence of almost finite generation. Let $M$ be an almost finitely generated $\mathcal O_K$-module, and let $a,b\in \mathcal O_K$ satisfy $v(b)>v(a).$ Then there exists a finitely generated submodule $N\subset M$ such that $bM\subset N\subset aM.$
		Indeed, applying almost finite generation to $b/a\in \mathfrak m$, we can find a finitely generated $\mathcal O_K$-module $F$ and maps $M\longrightarrow F\longrightarrow M$ whose composite is multiplication by $b/a$. If $P\subset M$ denotes the image of $F$ in $M$, then $P$ is finitely generated and contains $(b/a)M$. Hence $N:=aP$ has the desired property.
		
		Now choose an arbitrary $t\in \mathfrak m$. It suffices to prove that multiplication by $t$ kills $R^1\lim_n M_n$. Since the value group of $K$ is dense, we may choose elements $1=\varpi_0,\varpi_1,\varpi_2,\ldots \in \mathcal O_K$ such that $0=v(\varpi_0)<v(\varpi_1)<v(\varpi_2)<\cdots < v(t).$
		Applying the preceding observation to $M_n$, with $a=\varpi_{n-1}$ and $b=\varpi_n$, we obtain finitely generated submodules $N_n\subset M_n$ such that
		$$\varpi_n M_n\subset N_n\subset \varpi_{n-1}M_n .$$
		These submodules form an inverse subsystem. Indeed, writing $f_n:M_{n+1}\to M_n$ for the transition maps, we have
		$$f_n(N_{n+1})\subset f_n(\varpi_n M_{n+1})\subset \varpi_n M_n\subset N_n .$$
		Thus $(N_n)$ is an inverse system of finitely generated $\mathcal O_K$-modules.
		
		By Proposition \ref{vanishingalmostfg2}, $R^1\lim_n N_n\simeq^a 0$. Using the standard description
		$$R^1\lim_n M_n=\operatorname{coker}\left(\prod_n M_n\xrightarrow{D}\prod_n M_n\right),\qquad D((x_n))=(x_n-f_n(x_{n+1})),$$
		let a class in $R^1\lim_n M_n$ be represented by $(y_n)_n\in \prod_n M_n$. Since $v(t)>v(\varpi_n)$, we have $tM_n\subset \varpi_n M_n\subset N_n$ for every $n$. Hence $(ty_n)_n\in \prod_n N_n$. The almost vanishing of $R^1\lim_n N_n$ gives elements $x_n\in N_n$ such that $x_n-f_n(x_{n+1})=ty_n$ for all $n$. Therefore, the class represented by $(ty_n)_n$ is zero in $R^1\lim_n M_n$, as required.
	\end{proof}
	
	\subsection{General partially proper varieties} 
	
	In this subsection, we study the behavior of \'etale cohomology of integral structure sheaves for partially proper varieties. Although the conclusion of Theorem \ref{main} does not hold in this generality, we prove the following interesting result.
	
	\begin{thm} \label{stein}
		Let $X$ be a smooth partially proper space, and let $ \{U_n\}_{n\in\N}$ be an increasing sequence of quasi-compact adic spaces over $K$ such that $X=\bigcup_n U_n$. Then for any $i \geq 0$ and $m\geq 1$, $$R^1\lim_n H^i_{\et}(U_n, \mathcal{O}^+/p^m)$$ is almost zero.
	\end{thm}

	\begin{proof}
		We may assume that $U_n$ is strictly open in $U_{n+1}$ for each $n \geq 0$. Then \cite[Corollary 3.6]{scholze2018cohlubintate} shows that the image of $H^i_{\et}(U_{n+1}, \mathcal{O}^+/p^m)$ in $H^i_{\et}(U_n, \mathcal{O}^+/p^m)$ is an almost finitely generated $\mathcal O_K$-module for each $n \geq 0$. Then we can apply Theorem \ref{vanishingalmost} to conclude.	
	\end{proof}

	We record the following result, which shows that one cannot expect $R^1\lim$ to be strictly equal to zero, even for open balls.

	\begin{prop} \label{opendiscsection}
		Fix $d\geq 1$. For $r>0$, let $B(r)$ denote the open ball of dimension $d$ and radius $r$ over $K$. Then the following claims are equivalent:
		
		(1) $K$ is spherically complete.
		
		(2) For every $r>0$ and every increasing sequence $\{U_n\}_{n\in\N}$ of affinoid open subspaces such that $B(r)=\bigcup_n U_n$, we have $$R^1\lim_n \bigl(\mathcal{O}^+(U_n)/p^m\bigr)=0$$ for every $m\geq 1$.
		
		(3) For every $r>0$ and every increasing sequence $\{U_n\}_{n\in\N}$ of affinoid open subspaces such that $B(r)=\bigcup_n U_n$, we have $$R^1\lim_n \mathcal{O}^+(U_n)=0.$$
	\end{prop}

	\begin{proof}
		(1) $\Rightarrow$ (2): Fix $r>0$ and choose elements $\varpi_n\in K$ such that $|\varpi_n|$ strictly increases to $r$. Let $U_n:=\Spa(K\langle \varpi_n^{-1}T_1,...,\varpi_n^{-1}T_d\rangle, \mathcal{O}_K\langle \varpi_n^{-1}T_1,...,\varpi_n^{-1}T_d\rangle)$ be the closed ball of radius $|\varpi_n|$. By Lemma \ref{choice}, it suffices to show that $$R^1\lim_n\mathcal{O}_K[ \varpi_n^{-1}T_1,...,\varpi_n^{-1}T_d]/p^m=0.$$ One can check that the image of $\mathcal{O}^+(U_{n+1})/p^m$ in $\mathcal{O}^+(U_n)/p^m$ is a finite $\mathcal O_K$-module, and then one can apply Corollary \ref{linearvanishing}.

		(2) $\Rightarrow$ (3) follows from the derived Nakayama lemma, as $R^1\lim_n  \mathcal{O}^+(U_n)$ is derived $p$-adically complete, and $$(R^1\lim_n  \mathcal{O}^+(U_n))/p \simeq R^1\lim_n( \mathcal{O}^+(U_n)/p)=0.$$
		
		(3) $\Rightarrow$ (1): We adapt the nested-ball argument in \cite[Proposition 6]{lazard1962zeros}. Suppose, for a contradiction, that $K$ is not spherically complete. We can then choose a strictly decreasing sequence $r_n\in |K^\times|$ and centers $x_n\in K$ such that the closed balls $D(x_n,r_n)$ satisfy the following conditions for every $n\geq 1$:
		\begin{enumerate}
			\item[(a)] $D(x_n,r_n) \supset D(x_{n+1},r_{n+1})$;
			\item[(b)] $x_{n}\notin D(x_{n+1},r_{n+1})$;
			\item[(c)] $\bigcap_n D(x_n,r_n) =\varnothing.$
		\end{enumerate}
		Since $K$ is complete, $r:=\lim_n r_n>0$. Conditions (a) and (b) give $$r_{n+1}<|x_n-x_{n+1}|\leq r_n.$$
		
		Choose $\varpi_n\in K$ with $|\varpi_n|=r_n^{-1}$, and let $U_n$ be the closed ball of radius $|\varpi_n|$ centered at the origin. Then $\{U_n\}_{n\geq 1}$ is an increasing exhaustion of $B(r^{-1})$. Define $f_n:=(x_n-x_{n+1})T_1$. The above inequalities show that $f_n\in\mathcal O^+(U_n)$ but $f_n\notin\mathcal O^+(U_{n+1})$. By (3), there exists a sequence $g_n\in\mathcal O^+(U_n)$ satisfying $g_n-g_{n+1}=f_n$. Let $h_n$ be the coefficient of $T_1$ in $g_n$. Since $$g_1=f_1+\cdots+f_n+g_{n+1}=(x_1-x_{n+1})T_1+g_{n+1},$$ and $|h_{n+1}|\leq r_{n+1}$, we have $$|h_1-(x_1-x_{n+1})|\leq r_{n+1}.$$ Thus $x_1-h_1\in D(x_{n+1},r_{n+1})$ for every $n\geq 1$, contradicting (c). This completes the proof.
	\end{proof}

	\section{Applications and discussions}
	
	In this section, we discuss applications of our main results. We first prove a strengthened form of Riemann's extension theorem and compare it with its perfectoid analogue. We then study small pro-\'etale $G$-torsors on open varieties and discuss their relation to the local $p$-adic Simpson correspondence.

	\subsection{The first Riemann extension theorem} In this part, we take a close look at Theorems \ref{main} and \ref{mainlog}, and prove the following Hebbarkeitssatz.
	
	\begin{thm}[Riemann I] \label{riemann}
		Let $X$ be a normal affinoid adic space locally of finite type over $\Spa(K,K^+)$ and $D \subset X$ a Zariski-closed subset of $X$ with positive codimension everywhere such that $U:=X-D$ is smooth. Then for any increasing sequence $ \{U_n\}_{n\in\N}$ of affinoid adic spaces over $K$ such that $U=\bigcup_n U_n$, we have $$\mathcal{O}^+(U)=\mathcal{O}^+(X), \quad R^1\lim_n \mathcal{O}^+(U_n) \simeq^a 0.$$
	\end{thm}
	
	\begin{rem}
		The assumptions that $X$ is affinoid, $U$ is smooth and $U_n$ is affinoid for $n \in \N$ are not necessary for the first statement, but are necessary for proving the vanishing of $R^1\lim$.
	\end{rem}
	
	\begin{proof}
		The first statement is analytically local. Therefore, by the second Riemann extension theorem \cite[Theorem 10]{bartenwerfer1976riemannextension}, and because $X$ is normal, we may assume that $X$ and $D$ are smooth. The second statement on the vanishing of $R^1\lim$ is independent of $D$. Therefore, by the theorem of embedded desingularization \cite[Theorem 1.1.13]{temkin2018embeddedresolution}, we may assume that $X$ is smooth and $D$ is a strict normal crossing divisor. Thus, for both statements, we can assume that $X$ is smooth and $D$ is a strict normal crossing divisor. We endow $X$ with the log structure associated to $D$.
		
		We have $$R\lim_nR\lim_m H^0_{\et}(U_n,\mathcal{O}^+/p^m) \simeq^a R\lim_n \widehat{\mathcal{O}}^+(U_n) = R\lim_n {\mathcal{O}}^+(U_n),$$ where the first quasi-isomorphism follows from Proposition \ref{limitcoh} below, and the second quasi-isomorphism follows from \cite[Corollary 6.19]{scholze2013p} and \cite[Lemma 4.2]{scholze2013p}. On the other hand, by Theorem \ref{mainlog} and Proposition \ref{limitcohlog} below\footnote{$X$ is not affinoid in general, but by the theorem of embedded desingularization, $X$ is obtained by a sequence of smooth blow-ups of a normal affinoid variety $X'$. If we set $f:X\to X'$, then $f_*\mathcal{O}_X =\mathcal{O}_{X'}$. This moreover implies that $f_*\mathcal{O}_X^+ =\mathcal{O}_{X'}^+$. On the other hand, the relative version of the primitive comparison theorem shows that $f_*(\mathcal{O}_X^+/p^m)\simeq^a \mathcal{O}_{X'}^+/p^m.$ We can therefore apply Proposition \ref{limitcoh} to $X$.}, $$ R\lim_mR\lim_n H^0_{\et}(U_n,\mathcal{O}^+/p^m) \simeq^a R\lim_m H^0_{\ket}(X,\mathcal{O}^+/p^m)\simeq\widehat{\mathcal{O}}^+(X)=\mathcal{O}^+(X).$$
		Therefore, we have $$\mathcal{O}^+(U)=\lim_n \mathcal{O}^+(U_n)\simeq^a\mathcal{O}^+(X), \quad R^1\lim_n \mathcal{O}^+(U_n) \simeq^a 0.$$
		On the other hand, the fact that $\mathcal{O}^+(X)$ is almost isomorphic to $\mathcal{O}^+(U)$ means that for any $f\in \mathcal{O}^+(U)$ and $\epsilon >0$, $p^{\epsilon}f \in \mathcal{O}^+(X)$. This shows that $f \in \mathcal{O}^+(X)$. Therefore we have an (honest) isomorphism between $\mathcal{O}^+(X)$ and $\mathcal{O}^+(U)$. This finishes the proof of the Hebbarkeitssatz.
	\end{proof}
	
	\begin{rem}
		The same proof also shows that if we moreover assume that $D$ is a strict normal crossing divisor of $X$, then for $i\geq 1$, we have $$R^1\lim_n H^i_{\proet}(U_n,\widehat{\mathcal{O}}^+)\simeq^a0.$$
	\end{rem}

	\begin{prop} \label{limitcoh}
		Let $X$ be a smooth affinoid or smooth proper adic space over $\Spa(K,K^+).$ Then $$R^1\lim_m H^i_{\et}(X,\mathcal{O}^+/p^m) = 0 $$ for any $i \geq 0.$ In particular, for $i\geq 0$ we have $$H^i_{\proet}(X,\widehat{\mathcal{O}}^+) \simeq \lim_m H^i_{\et}(X,\mathcal{O}^+/p^m).$$
	\end{prop}

	\begin{proof}
		We first consider the case where $X$ is smooth affinoid. The exact sequence on $X_{\proet}$ $$0 \to \widehat{\mathcal{O}}^+ \xrightarrow{p^m} \widehat{\mathcal{O}}^+ \to \widehat{\mathcal{O}}^+/p^m \to 0$$ gives $$0 \to  H^i_{\proet}(X,\widehat{\mathcal{O}}^+)/p^m \to  H^i_{\et}(X,{\mathcal{O}}^+/p^m) \to H^{i+1}_{\proet}(X, \widehat{\mathcal{O}}^+)[p^m] \to 0,$$ which gives $$R^1\lim_m H^i_{\et}(X,\mathcal{O}^+/p^m) \simeq R^1\lim_m (H^{i+1}_{\proet}(X, \widehat{\mathcal{O}}^+)[p^m]).$$
		To prove the proposition, it suffices to show that for any $i \geq 0$, the cohomology group $H^{i}_{\proet}(X, \widehat{\mathcal{O}}^+)$ has bounded $p$-torsion when $X$ is smooth affinoid. 
		
	    Let $X=\Spa(R,R^+).$ Take a finite cover $X^0=\Spa(R_0,R^{+}_0)\to X$ by rational subdomains such that $X^0$ is a finite disjoint union of toric charts (i.e., $X^0=\bigsqcup U_s$, where each $U_s$ admits an \'etale map $U_s \to \mathbb T^d:=\Spa(K\langle T_1^{\pm 1},...,T_d^{\pm 1}\rangle,K^+\langle T_1^{\pm 1},...,T_d^{\pm 1}\rangle)$ which can be written as a composition of rational embeddings and finite \'etale maps). The open cover $X^0\to X$ gives a \v Cech nerve $X^{\bullet}\to X.$ The \v Cech-to-cohomology spectral sequence gives $$E_1^{i,j}:=H^j_{\proet}(X^i,\widehat{\mathcal{O}}^+)\Rightarrow H^{i+j}_{\proet}(X,\widehat{\mathcal{O}}^+).$$
		
		We claim that $E_2^{i,j}$ has bounded torsion for each pair $i,j\geq 0$, and that $E_2^{i,j}$ is a torsion group for $i >0.$ For $j=0$, this follows from (the proof of) \cite[Proposition 6.10.(iii)]{scholze2012perfectoid}. In general, write $X^i=\Spa(R_i,R^{+}_i)$ for $i \geq 1$. For each pair $i,j\geq 0$, according to \cite[Lemma 2.14.(iii)]{heuer2026moduli} (which is basically a combination of \cite[Lemma 4.5]{scholze2013p} and \cite[Lemma 5.5]{scholze2013p}), the kernel and cokernel of $$f_{i,j}: \bigwedge^j (R_i^+)^d \to E_1^{i,j}$$ are killed by some $p^c$ ($c$ only depends on $i$). Fix $j$, and let $C^{\bullet,j}$ be the cone of $f_{\bullet,j}$. Then the cohomology of $C^{\bullet,j}$ in each degree is killed by some power of $p$. Applying the argument of \cite[Proposition 6.10.(iii)]{scholze2012perfectoid} again, we see that $H^i(\bigwedge^j (R_{\bullet}^+)^d)$ also has bounded torsion for $i\geq 0$, and is a torsion group for $i >0.$ Then the long exact sequence $$\dots \to H^{i-1}(C^{\bullet,j})\to  H^i(\bigwedge^j (R_{\bullet}^+)^d) \to E_2^{i,j} \to H^i(C^{\bullet,j}) \to \dots$$ shows that $E_2^{i,j}$ has bounded torsion for $i\geq 0$, and is a torsion group for $i >0.$ This completes the proof of our claim.
		
		This implies that on each page $a\geq 2$, $E_a^{i,j}$ has bounded torsion for each $i,j\geq 0$, and is a torsion group for $i >0.$ Therefore the cohomology group $H^{i}_{\proet}(X, \widehat{\mathcal{O}}^+)$ has bounded $p$-torsion for each $i\geq 0$, as $H^{i}_{\proet}(X,\widehat{\mathcal{O}}^+)$ is obtained from the groups $E_a^{i-k,k}$ for $k \geq 0$ and some $a\in \N$ by successively taking extensions. This completes the proof when $X$ is affinoid.
		
		When $X$ is smooth proper, by repeatedly using the exact sequences $$0 \to \widehat{\mathcal{O}}^+/p^m \to \widehat{\mathcal{O}}^+/p^{m+1} \to \widehat{\mathcal{O}}^+/p \to 0$$ and the primitive comparison theorem \cite[Theorem 5.1]{scholze2013p}, one sees that $H^{i}_{\proet}(X, \widehat{\mathcal{O}}^+/p^m)$ is an almost finitely generated $\mathcal O_K$-module for each $m \geq 1$ by \cite[Proposition 2.6]{scholze2013p}. Then Theorem \ref{vanishingalmost} shows that $$R^1\lim_mH^{i}_{\proet}(X, \widehat{\mathcal{O}}^+/p^m)\simeq^a 0.$$ This completes the proof. 
	\end{proof}
	
	The same proof goes through in the logarithmic setting.
	
	\begin{prop} \label{limitcohlog}
		Let $X$ be an fs log adic space that is log smooth and either affinoid or proper over $\Spa(K,K^+).$ Then $$R^1\lim_m H^i_{\ket}(X,\mathcal{O}^+/p^m) = 0 $$ for any $i \geq 0.$ In particular, for $i\geq 0$ we have $$H^i_{\proket}(X,\widehat{\mathcal{O}}^+) \simeq \lim_m H^i_{\ket}(X,\mathcal{O}^+/p^m).$$
	\end{prop}
	
	An interesting corollary of the proof of Proposition \ref{limitcoh} is the uniform boundedness of torsion in the \'etale cohomology of the integral structure sheaf.
	
	\begin{cor}
		Let $X$ be a smooth affinoid adic space over $\Spa(K,K^+)$ of dimension $d$. Then there is some $c>0$ such that $p^c\cdot H^1_{\et}(X,{\mathcal{O}}^+)=0.$ Moreover, for $2 \leq i\leq d$, the group $H^i_{\et}(X,{\mathcal{O}}^+)$ is always unbounded torsion.
	\end{cor}
	
	\begin{proof}
		Tate's acyclicity theorem shows that for $i>0$ we have $H^i_{\et}(X,{\mathcal{O}}^+)[1/p]=H^i_{\et}(X,{\mathcal{O}})=0$, so $H^i_{\et}(X,{\mathcal{O}}^+)$ is torsion. On the other hand, \cite[Proposition 3.23]{scholze2013perfectoidspacessurvey} implies $\nu_*\widehat{\mathcal{O}}^+={\mathcal{O}}^+,$ where $\nu$ is the natural morphism of sites $X_{\proet}\to X_{\et}.$ Then the Leray spectral sequence gives an injection $H^1_{\et}(X,{\mathcal{O}}^+) \hookrightarrow H^1_{\proet}(X,\widehat{\mathcal{O}}^+),$ which completes the proof. For the second statement,  use\footnote{Use $\mathcal \nu^{-1}(O_X/\mathcal O_X^+)\simeq \widehat{\mathcal O}_X/\widehat{\mathcal O}_X^+$ and \cite[Corollary 3.17]{scholze2013p}.} $R^i\nu_*\widehat{\mathcal{O}}^+=\Omega_X^i(-i)$ for $i\geq 1$, and the Leray spectral sequence $$E_2^{i,j}:=H^j_{\et}(X,R^i\nu_*\widehat{\mathcal{O}}^+)\Rightarrow H^{i+j}_{\proet}(X,\widehat{\mathcal{O}}^+) $$ to conclude.
	\end{proof}

	\subsection{Hebbarkeitssatz for perfectoid algebras} \label{riemannperf}
	
	We compare our results with the ``quantitative form of Scholze's Riemann extension theorem'' proved in \cite[Theorem 4.2]{bhatt2018directsumconj}. In the situation below, the calculation of Subsection \ref{disc} gives a stronger extension statement in degree zero after taking the inverse limit, although higher cohomology need not almost vanish.
	
	Suppose that $X=\Spa(R,R^+)$ is an affinoid fs log adic space of finite type over $\Spa(K,K^+)$ with a strictly \'etale morphism $X \to \mathbb B^d:=\Spa(K\langle T_1,...,T_d\rangle, K^+\langle T_1,...,T_d\rangle)$, which can be written as a composite of rational embeddings and finite \'etale maps. In this subsection, we use the pure $p$-power tower of Subsection \ref{disc}. Pulling it back to $X$ gives $$\widetilde X:=\Spa(R_{\infty},R^+_{\infty})\longrightarrow X,$$ a Galois pro-finite Kummer \'etale cover with group $\Gamma:=\Z_p^d$, whose underlying space is affinoid perfectoid.
	
	Denote by $D \subset X$ the closed subset defined by $g:=T_1\cdots T_d=0$, and denote by $\widetilde{D}$ the pullback of $D$ along $\widetilde{X}\to X.$ Let $U:=X-D$ and $\widetilde{U}:=\widetilde{X}-\widetilde{D}.$ The map $\widetilde{U}\to U$ is a Galois pro-finite \'etale cover with Galois group $\Gamma$. We note that $U$ is a quasi-Stein space: if we set $V_n:=\Spa(S_n,S_n^+),$ where $$S_n:=R\left\langle \dfrac{p^n}{g}\right\rangle,$$ then $U$ is the increasing union of the $V_n$. We also denote by $\widetilde{V}_n=\Spa(S_{n,\infty},S_{n,\infty}^+)$ the pullback of $V_n$ along $\widetilde{U}\to U.$ The ``quantitative form of Scholze's Riemann extension theorem'' is stated as follows.
	
	\begin{thm}{{\cite[Theorem 4.2]{bhatt2018directsumconj}}} \label{qriemann}
		For $m\geq 1$, the map from the constant pro-system with value $R^+_{\infty}/p^m$ to the pro-system $\{S_{n,\infty}^+/p^m\}_{n\in \N}$ is almost injective with respect to $p^{1/p^{\infty}}$, and is almost surjective with respect to $(pg)^{1/p^{\infty}}$. In particular, the map $R^+_{\infty}/p^m \to H^0_{\proet}(\widetilde{U},\mathcal{O}^+_{\widetilde{U}}/p^m)$ is an almost isomorphism with respect to $(pg)^{1/p^{\infty}}$.
	\end{thm}
	
	One might try to prove Theorem \ref{mainlog} from Theorem \ref{qriemann} using the Cartan--Leray spectral sequence. In the present setting, this requires an almost isomorphism of pro-systems
	$$\{H^i_{\cts}(\Gamma,S_{n,\infty}^+/p^m)\}_{n\in\N}\simeq^a c\bigl(H^i_{\cts}(\Gamma,R_\infty^+/p^m)\bigr),$$
	with respect to $p^{1/p^{\infty}}$. Theorem \ref{qriemann} already compares the coefficient pro-systems, but its almost surjectivity is only with respect to $(pg)^{1/p^{\infty}}$. Taking group cohomology does not automatically strengthen this almost condition. For example, consider the module $K^+/p$ over the standard coordinate algebra, with all $T_i^{1/p^n}$ acting by zero and with the trivial $\Gamma$-action. It is almost zero with respect to $(pg)^{1/p^{\infty}}$, but its $\Gamma$-invariants are $K^+/p$ itself, which is not almost zero with respect to $p^{1/p^{\infty}}$.

	In our situation, however, the weight decomposition of Subsection \ref{disc} gives this stronger comparison after taking the ordinary inverse limit, as follows.

	\begin{prop} \label{almostperf}
		For every $m\geq1$, the natural map $$R^+_{\infty}/p^m \longrightarrow H^0_{\proet}(\widetilde U,\mathcal O^+_{\widetilde U}/p^m)$$ is an almost isomorphism with respect to $p^{1/p^{\infty}}$.
	\end{prop}

	\begin{proof}
		We only give a sketch, as the proof is similar to the proof in Section \ref{proofmain} with a different settings. We start from a tubular neighborhood compatible with the given coordinates. Choose a topologically nilpotent $\varpi\in K$ with compatible $p$-power roots. After shrinking the neighborhood and rescaling the coordinates, we may moreover assume that its pullback to the tower is $\Spa(B,B^\circ)$, where $$B=C\langle z_1^{1/p^\infty},\ldots,z_r^{1/p^\infty}\rangle$$ for a perfectoid $K$-algebra $C$.

		Let $W$ be the complement of $z_1\cdots z_r=0$ in $\Spa(B,B^\circ)$, and set $W_n:=\{|z_1\cdots z_r|\geq|\varpi|^n\}$, so that $W=\bigcup_n W_n$. We use the decomposition of Subsection \ref{disc}, with $C^\circ/p^m$ in place of $\mathcal O_K/p^m$. On the summand indexed by $\mathbf w$, the transition map from $N$ to $n$ is multiplication by $\varpi^{(N-n)\max(\mathbf w)}$. For each fixed $\mathbf w$, this map is zero modulo $p^m$ for sufficiently large $N$. Thus these summands have zero inverse limit, and only $B^\circ/p^m$ remains. Since each $W_n$ is affinoid perfectoid, we obtain
		$$B^\circ/p^m\simeq^a\lim_n\mathcal O^+(W_n)/p^m\simeq^a H^0_{\proet}(W,\mathcal O^+/p^m).$$ This prove the proposition for a tubular neighborhood, and the general case follows by descent\footnote{We don't have pro-isomorphism, but it is not a problem as we only work with global sections.} as in Subsection \ref{descent}.
	\end{proof}
	
	Here the required transition step may depend on the weight. This suffices for the ordinary inverse limit, but does not show that the additional weight pro-system is almost zero. In particular, Proposition \ref{almostperf} does not give a comparison of derived inverse limits: the $R^1\lim$ term need not almost vanish, as Proposition \ref{failperfectoid} shows. Thus computing the cohomology of $U$ by Cartan--Leray would not by itself prove the pro-system statements of Theorems \ref{main} and \ref{mainlog}. On the other hand, we still have
	$$H_{\proet}^{i}(\widetilde{U}, \mathcal{O}_{\widetilde{U}}^{+}/p^m) \simeq^a
	\begin{cases} \lim_n(S_{n,\infty}^{+}/p^m) & \text{if } i=0 \\ R^1 \lim_n(S_{n,\infty}^{+}/p^m) & \text{if } i=1 \\ 0 & \text{if } i \ge 2.
	\end{cases}$$
	Proposition \ref{failperfectoid} below shows that the degree-one term need not be almost zero with respect to $p^{1/p^{\infty}}$. To apply the Cartan--Leray spectral sequence without choosing a topology on this higher cohomology group, we use the condensed formulation of \cite[Proposition 4.12]{bosco2023padicproetalecohomologydrinfeld}:
    $$E_{2}^{i,j} := H_{\cond}^{i}(\Gamma, H_{\proet}^{j}(\widetilde{U}, \underline{\mathcal{O}^+_{\widetilde{U}}}/p^m)) \Rightarrow H_{\proet}^{i+j}(U, \underline{\mathcal{O}^+_{U}}/p^m).$$
    This gives an almost exact sequence
    $$\dots   \to E_2^{i,0} \to H_{\proet}^{i}(U, \underline{\mathcal{O}^+_{U}}/p^m) \to E_2^{i-1,1} \to E_2^{i+1,0} \to H_{\proet}^{i+1}(U, \underline{\mathcal{O}^+_{U}}/p^m) \to \dots$$
    According to \cite[Remark 2.12 and Proposition B.2]{bosco2023padicproetalecohomologydrinfeld}, we have $$H_{\proet}^{i}(U, \underline{\mathcal{O}^+_{U}}/p^m)(*)=H_{\proet}^{i}(U, {\mathcal{O}^+_{U}}/p^m), \quad E_2^{i,0}(*)=H_{\cts}^{i}(\Gamma, H_{\proet}^{0}(\widetilde{U}, \mathcal{O}^+_{\widetilde{U}}/p^m))$$ for all $i\geq 0$, where $M(*)$ denotes the underlying abelian group of a condensed abelian group $M$, and the degree-zero sections carry the inverse-limit topology. Proposition \ref{almostperf} easily extends to the condensed setting, giving the following corollary.
	
	\begin{cor}
		For every $i\geq0$ and $m\geq1$, the natural map $$H^i_{\cts}(\Gamma,R_\infty^+/p^m)\longrightarrow H^i_{\cts}(\Gamma,H^0_{\proet}(\widetilde U,\mathcal O^+_{\widetilde U}/p^m))$$ is an almost isomorphism with respect to $p^{1/p^{\infty}}$. Here the source coefficient module is discrete, and the target coefficient module has the inverse-limit topology. In particular, we have for every $i\geq0$ and $m\geq1$, $$E_2^{i,1} \simeq^a 0.$$
	\end{cor}



	Thus the degree-zero sections satisfy the stronger extension statement of Proposition \ref{almostperf}, but the higher cohomology still gives an obstruction to extending Theorem \ref{main} to perfectoid spaces. We finish with the promised example.
	
	\begin{prop} \label{failperfectoid}
		Consider the map $$\widetilde{X}:=\Spa(K\langle T^{1/p^{\infty}}\rangle, \mathcal O_K\langle T^{1/p^{\infty}}\rangle) \to X:=\Spa(K\langle T\rangle, \mathcal O_K\langle T\rangle).$$ For $n\geq 0$, let $U_n:=\Spa(K\langle T,p^n/T\rangle, \mathcal O_K\langle T,p^n/T\rangle)$ and $\widetilde{U}_n$ be the pullback of $U_n$ along $\widetilde{X} \to X$. Then (in the almost sense) $$R^1\lim_n H^0_{\et}(\widetilde{U}_n, \mathcal{O}^+/p)\neq 0.$$ 
	\end{prop}

	\begin{proof}
		As $\widetilde{U}_n$ is an affinoid perfectoid space over $K$ for $n \geq 0$, we have $$H^0_{\et}(\widetilde{U}_n, \mathcal{O}^+/p)\simeq^a\mathcal{O}_K/p[T^{1/p^{\infty}},(p^n/T)^{1/p^{\infty}}].$$
		Set $f_n(T) := (p^n/T)^{1/p^{n}} \in H^0_{\et}(\widetilde{U}_n, \mathcal{O}^+/p)$. It suffices to prove that there does not exist a sequence of functions $g_n \in H^0_{\et}(\widetilde{U}_n, \mathcal{O}^+/p)$ satisfying $g_n - g_{n+1} = f_n$.
		
		Suppose the opposite. We write $g_0=\sum_{w}a_wT^w$, where $w \in \Z[1/p]$, $a_w \in \mathcal{O}_K/p$, and only finitely many $a_w$ are nonzero. Then there exists $N \in \N$ such that $a_{-1/p^n}=0$ for any $n\geq N.$ However, as $$g_0=\sum_{n=0}^{N} f_n+g_{N+1}=\sum_{n=0}^{N} (p^n/T)^{1/p^{n}}+g_{N+1},$$ and $(p^{N}/T)^{1/p^{N}}\notin H^0_{\et}(\widetilde{U}_{N+1}, \mathcal{O}^+/p),$ we have $a_{-1/p^N} \neq 0$, which leads to a contradiction. This completes the proof.
	\end{proof}
	
	\begin{rem}
		On the other hand, Theorem \ref{qriemann} implies that $R^1\lim_n H^0_{\et}(\widetilde{U}_n, \mathcal{O}^+/p)$ is almost zero with respect to $(pT)^{1/p^{\infty}}.$
	\end{rem}
	
	\subsection{Applications to pro-\'etale \texorpdfstring{$G$}{G}-torsors} \label{torsors}
	
	In this subsection, we apply Theorems \ref{main} and \ref{mainlog} to study pro-\'etale (or equivalently, $v$) $G$-torsors over open rigid spaces. We will also discuss their relation to local correspondences.
	
	We refer the reader to \cite[Section 3]{heuer2026gtorsor} for the notation concerning rigid groups and their associated Lie algebras. In particular, rigid analytic groups are also regarded as sheaves on the relevant $v$-sites. More explicitly, whenever no confusion can arise, we freely identify a rigid group $G$ over $K$ with its associated pro-\'etale sheaf or pro-Kummer \'etale sheaf over $K$, depending on the context. Let $X$ be a smooth adic space (resp. an fs log adic space that is log smooth) over $\Spa(K,K^+)$. Let $\tau \in\{\proet,\proket\}$.
	
	Assume in this subsection that $G$ has good reduction. Thus we fix a smooth formal model $\mathcal G$ of $G$ over $\operatorname{Spf}(K^+)$. Let $\mathfrak g^+$ be the Lie algebra of $\mathcal G$. It is a finite projective, hence free, $K^+$-module, and we regard it as a canonical open $K^+$-lattice $\mathfrak g^+ \subset \mathfrak g:=\mathfrak g^+ \otimes_{K^+} K$. We recall some notation and properties that will be used in this subsection.
	
	\begin{lem}\cite[Lemma 4.13]{heuer2026gtorsor}
		The sheaf $G$ on $X_{\tau}$ is the analytic sheafification of the functor
		$$ (R,R^+) \longmapsto \mathcal G(R^+). $$
	\end{lem}
	
	This gives an integral approximation of $G$ by reduction modulo the ideals $p^k\mathfrak m$.
	
	\begin{defn}\cite[Definition 4.14]{heuer2026gtorsor} \label{smallsubgroups}
		For $0\leq k\in v(K^\times)$, let $\overline G_k$ be the sheaf on $X_{\tau}$ obtained by \'etale sheafification of the presheaf
		$$ Y \longmapsto \mathcal G\bigl(\mathcal O^+(Y)/p^k\mathfrak m\,\mathcal O^+(Y)\bigr). $$
		We denote by $G_k\subseteq G$ the kernel of the natural morphism $G\to \overline G_k$. Equivalently, one has an exact sequence of sheaves
		$$ 1\longrightarrow G_k\longrightarrow G\longrightarrow \overline G_k\to 1. $$
		After passing to sheaves on $X_{\tau}$, we also write $\overline G_k:=G/G_k$.
	\end{defn}
	
	The subgroups $G_k$ form a system of small integral neighbourhoods of the identity. On the Lie algebra side we use the following parallel notation.
	
	\begin{defn}\cite[Definition 4.20]{heuer2026gtorsor}
		For $0\leq k\in v(K^\times)$, define
		$ \mathfrak g^+_k:=p^k\mathfrak m\,\mathfrak g^+\subseteq \mathfrak g. $
		This is represented by an open rigid subgroup of the additive rigid group $\mathfrak g$. We also set
		$ \overline{\mathfrak g}^{\, +}_k:=\mathfrak g^+/\mathfrak g^+_k $
		as a sheaf on $X_{\tau}$. If $d=\dim G$, then, since $\mathfrak g^+$ is free over $K^+$, there is a non-canonical isomorphism of $\mathcal O^+$-modules
		$ \overline{\mathfrak g}^{\, +}_k \simeq \bigl(\mathcal O^+/p^k\mathfrak m\,\mathcal O^+\bigr)^d. $
	\end{defn}
	
	For sufficiently small neighbourhoods of the origin, the exponential identifies the Lie-theoretic and group-theoretic filtrations.
	
	\begin{lem}\cite[Lemma 4.21]{heuer2026gtorsor}
		There is a constant $\alpha\geq 1/(p-1)$ such that for every $k\in v(K^\times)$ with $k>\alpha$, the exponential of $G$ is defined on $\mathfrak g^+_k$ and induces an isomorphism of rigid spaces $$ \exp:\mathfrak g^+_k \xrightarrow{\sim} G_k. $$ We denote its inverse by $\log$ whenever $k>\alpha$.
	\end{lem}
	
	We write $\operatorname{Tors}_{G_k}$ for the category of $G_k$-torsors and $H^1(-,G_k)$ for its pointed set of isomorphism classes. Theorems \ref{main} and \ref{mainlog}, together with Propositions \ref{limitcoh} and \ref{limitcohlog} and the successive additive quotients of the small subgroups \cite[Lemma 4.24]{heuer2026gtorsor}, give the following result.
	
	\begin{thm} \label{gtorsor}
		We fix $k>\alpha$ as above. Let $X$ be an fs log adic space that is log smooth over $\Spa(K,K^+)$. Denote by $U:=X_{\operatorname{tr}}$ the maximal open subset of $X$ on which the induced log structure is trivial. Then restriction induces an equivalence of categories $$\operatorname{Tors}_{G_k}(X_{\proket}) \xrightarrow{\sim} \operatorname{Tors}_{G_k}(U_{\proet}).$$ In particular, any $G_k$-torsor on $U_{\proet}$ extends uniquely to a $G_k$-torsor on $X_{\proket}$.
		
		Moreover, if $X$ is affinoid or proper and $\{U_n\}_{n\in\N}$ is an increasing sequence of affinoid open adic subspaces over $K$ such that $U=\bigcup_n U_n$, then restriction induces a bijection of pointed sets $$H^1_{\proet}(U,G_k) \xrightarrow{\sim} \lim_n H^1_{\proet}(U_n,G_k).$$ Thus the isomorphism class of a $G_k$-torsor on $U_{\proet}$ is determined by, and can be recovered from, a compatible system of isomorphism classes on the $U_{n,\proet}$.
	\end{thm}
	
	\begin{rem}
		The corresponding statements need not hold when $G_k$ is replaced by an arbitrary rigid analytic group. The first comparison already fails for $G=\mathbb G_a$. An analogous inverse-limit comparison in degree one also fails for general open rigid spaces: the nontrivial line bundles on open discs constructed in \cite[Proposition 6]{lazard1962zeros} give counterexamples for $G=\mathbb G_m$, since their restrictions to closed subdiscs are trivial.
	\end{rem}
	
	The above discussion also shows the difficulty in extending the $p$-adic Simpson correspondence to open varieties. As in the proper case \cite[Theorem 1.1]{heuer2024padicsimpsoncorrespondencesmooth}, we would like to construct a correspondence between pro-\'etale vector bundles and Higgs bundles. However, the preceding remark shows that even line bundles on open discs behave differently as the base field varies.

	For another example, let $X$ be the analytification of a smooth algebraic variety, and choose a smooth proper compactification $Y$ with a strict normal crossing boundary $D:=Y-X$. We endow $Y$ with the log structure associated to $D$, so that $Y$ is log smooth. In this case, there are two versions of the $p$-adic Simpson correspondence to consider: the usual correspondence on $X$, and a correspondence between pro-Kummer \'etale vector bundles and log Higgs bundles on $Y$.

	One key ingredient in \cite{heuer2024padicsimpsoncorrespondencesmooth} is the small local correspondence constructed in \cite[Theorem 6.5]{heuer2026moduli}. Roughly speaking, a small vector bundle is one that comes from a torsor under a sufficiently small subgroup of $\mathrm{GL}_r$, and there is a corresponding notion of a small Higgs bundle. Locally on a toric chart, the two categories are equivalent. The same proof works in the logarithmic setting and gives an equivalence between small pro-Kummer \'etale vector bundles and small log Higgs bundles. One should therefore be able to construct an equivalence between pro-Kummer \'etale vector bundles and log Higgs bundles on $Y$ from these local correspondences, as in the proper case.

	On the other hand, Theorem \ref{gtorsor} shows that, for each fixed $k>\alpha$, the category of $G_k$-torsors on $X_{\proet}$ is equivalent to the category of $G_k$-torsors on $Y_{\proket}$. Thus these small torsors on the open variety already extend to the logarithmic compactification. Therefore, for varieties which are not quasi-compact, we lose too much information when we restrict to the small case to study the correspondence on the open variety.


	\bibliography{ref}
	
	\bibliographystyle{alpha}
	
\end{document}